\documentclass[12pt,oneside,a4paper,reqno]{amsart}
\usepackage{latexsym,amsxtra,amsthm, amsmath, amsfonts, amscd, amssymb, verbatim}
\usepackage{amsfonts}
\usepackage[portrait, top=2.5 cm, bottom=3cm, left=3.5cm, right=2cm] {geometry} 

\usepackage{url}

\usepackage{varioref}
\usepackage{fancyhdr}
\usepackage{longtable}
\usepackage{fancybox}
\usepackage{color, graphicx}
\usepackage{verbatim}
\usepackage{setspace}
\usepackage{enumitem}

\usepackage[latin1]{inputenc}
\usepackage[english]{babel}

\newtheorem{prop} {Proposition} [section]
\newtheorem{thm}[prop] {Theorem}

\newtheorem*{defi-non}{Definition}

\newtheorem{defi and lem}[prop] {Definition/Lemma}
\newtheorem{cor}[prop]{Corollary}
\newtheorem{rem}[prop]{Remark}

\newtheorem{prop-def}[prop]{Proposition-Definition}

\begin{document}

\title[Discriminants of Cubic Curves and Determinantal Representations]{Discriminants of Cubic Curves and Determinantal Representations}

\author{Manh Hung Tran}
\address{Ericsson, Stockholm}\email{tranmanhhungkhtn@gmail.com}

\maketitle

\begin{abstract}
The discriminant of a smooth plane cubic curve over the complex numbers can be written as a product of theta functions. This provides an important connection between algebraic and analytic objects. In this paper, we perform a new approach to obtain this classical result by using determinantal representations. More precisely, one can represent a non-singular cubic form as the determinant of a matrix whose elements are linear forms. Theta functions naturally appear in this representation and thus in the discriminant of the cubic.

\end{abstract}

\tableofcontents

\section{Introduction}\label{section introduction}

The discriminant of a plane cubic curve is a polynomial of degree 12 in coefficients of the cubic with 2040 monomials (see \cite[p. 4]{Gelfand}). But over $\mathbb{C}$, we have short expressions in terms of theta constants. Consider the classical case where our smooth projective cubic curve $C$ is defined by the affine Weierstrass equation:
$$y^2=4x^3 - g_2x - g_3.$$
Using the Weierstrass parametrization, there exists a unique lattice $\Lambda=\omega_1 \mathbb{Z}+\omega_2 \mathbb{Z}$ with some complex numbers $\omega_1,\omega_2$ such that Im$(\omega_2/\omega_1)>0$ and $C(\mathbb{C})\cong \mathbb{C}/\Lambda$. Let $\tau:=\omega_2/\omega_1$ and apply the discriminant formula $\Delta=2^{12}(g_2^3 - 27g_3^2)$ from \cite[p. 367-368]{Villegas}, we have that
\begin{equation}\label{Disc Weierstrass}
\Delta=2^{16}\left(\frac{\pi}{\omega_1}\right)^{12}(\theta_2(0,\tau)\theta_3(0,\tau)\theta_4(0,\tau))^8.
\end{equation}
Here $\theta_2,\theta_3$ and $\theta_4$ are the three even Jacobi theta functions. The details of theta functions will be described in Section \ref{section det rep of elliptic curves}.

We want to study the above discriminant formula with a new approach using determinantal representations. For a homogeneous polynomial $\phi$, we construct a matrix $U$ whose elements are linear forms such that we can write $\phi=\lambda\det(U)$ for some constant $\lambda\neq 0$. The study of $\phi$ has thus been moved to the study of the matrix $U$. In general, only plane curves and quadratic, cubic surfaces admit a determinantal representation as confirmed in \cite{Dick}. The reader can have a look at \cite{Beau} for a general discussion of this topic.

Starting with Weierstrass cubics, we find theta functions in their determinantal representations as well as in the discriminants. Denote by
\begin{equation}\label{abc}
a=\theta_{2}(0,\tau), \text{ } b=\theta_{3}(0,\tau), \text{ } c=\theta_{4}(0,\tau)
\end{equation}
the theta constants. We will prove the following:
\begin{thm}\label{theorem det rep of elliptic curves}
Let $C_{\phi}$ be a smooth curve given by the Weierstrass form
$$\phi(x,y,z)=y^2z-4x^3+g_2xz^2+g_3z^3,$$
where $g_2$ and $g_3$ belong to a field $K$. Then $\phi$ admits determinantal representations
$$\left(
  \begin{array}{cccc}
    2x+tz & y+dz & (3t^2-g_2)z \\
    0 & x-tz & y-dz \\
    z & 0 & -2x-tz \\
  \end{array}
\right),$$
with $t,d\in \overline{K}$ being arbitrary such that $d^2=4t^3-g_2t-g_3$. When $K=\mathbb{C}$, there is a natural choice for $t,d$ which produces a determinantal representation for $\phi$ in terms of theta constants as follows
$$\left(
  \begin{array}{ccc}
    2x-\frac{\pi^2}{3\omega_1^2}(a^4+b^4)z & y & -(\frac{\pi}{\omega_1})^4 c^8 z \\
    0 & x+\frac{\pi^2}{3\omega_1^2}(a^4+b^4)z & y \\
    z & 0 & -2x+\frac{\pi^2}{3\omega_1^2}(a^4+b^4)z \\
  \end{array}
\right).$$
Here the even theta constants $a,b,c$ were defined as in \eqref{abc}.
\end{thm}

The first part of this theorem uses the method in \cite[Section 2]{Vinnikov} where the author established similar representations for other type of Weierstrass equations of the form $y^2z=x(x+\vartheta_1 z)(x+\vartheta_2 z)$ with some constants $\vartheta_1,\vartheta_2\in K$. The discriminant formula \eqref{Disc Weierstrass} is then a consequence of the second part of this theorem using resultant as in Section \ref{section det rep of elliptic curves}.

Our goal is to study this phenomena for general smooth cubic curves using determinantal representations. One can actually provide determinantal representations for any non-rational complex plane curve by using a result in \cite{Ball} as we will see later in Section \ref{section det rep of complex curves}. This deduces in particular the formula of the discriminant of plane cubics by using resultant. Since a cubic curve $C_{\phi}$ over $\mathbb{C}$ always has a flex point, it can be transformed to a Weierstrass form after a linear coordinate change $M$ (see \cite[Section 4.4]{Cremona transform}). The resulting Weierstrass form is isomorphic to $\mathbb{C}/\Lambda$ for a unique lattice $\Lambda$ coming the Weierstrass parametrization. Write $\Lambda=\omega_1\mathbb{Z}+\omega_2 \mathbb{Z}$ for some $\omega_1,\omega_2\in \mathbb{C}$ satisfying Im$(\omega_2/\omega_1)>0$ and denote by $\tau=\omega_2/\omega_1$, we will prove the following result:

\begin{thm}\label{theorem intro disc cubic theta}
Let $C_{\phi}$ be a smooth plane cubic curve over $\mathbb{C}$ defined by a cubic form $\phi$ and $\Delta_{\phi}$ be the discriminant of $\phi$, we have
\begin{equation}\label{disc cubic}
\Delta_{\phi}=\frac{2^{16}}{\det (M)^{12}}\left(\frac{\pi}{\omega_1}\right)^{12}(abc)^8,
\end{equation}
where $a,b,c$ were defined as in \eqref{abc}.
\end{thm}

This result is known from invariant theory since the discriminant of a smooth plane cubic curve is an invariant of weight 12 of the cubic form defining the curve (see Fisher \cite[Theorem 4.4]{Fisher} for instance). However, the above approach with determinantal representations is new. Theta functions naturally appear in these representations and thus a new way to compute the discriminants of cubic curves is obtained.

The discriminant formulae \eqref{Disc Weierstrass} and \eqref{disc cubic} are remarkable since they provide a connection between algebraic (discriminants) and analytic (theta functions) objects. In fact, there is a similar formula in the case of quartic curves studied by Klein \cite[p. 72]{Klein}.

The organization of this paper is as follows: Section \ref{section det rep of elliptic curves} provides an overview on theta functions, smooth cubic curves in Weierstrass forms and a proof of Theorem \ref{theorem det rep of elliptic curves}. A short introduction on Riemann surfaces is presented in Section \ref{section det rep of complex curves} and then we construct there determinantal representations for non-rational complex plane curves (see Theorem \ref{theorem det rep of complex curves}). As a consequence, a new proof for the discriminant formula of smooth plane cubic curves is accomplished in Section \ref{section disc of plane cubic curves}. In other words, we prove in this section Theorem \ref{theorem intro disc cubic theta}. An overview to Klein's formula and a known result on determinantal representations of plane quartics are described in the last section.

\section{Determinantal representations of Weierstrass cubics}\label{section det rep of elliptic curves}

We will in this section study discriminants of smooth curves in Weierstrass form and provide a proof to Theorem \ref{theorem det rep of elliptic curves}. Consider a smooth curve $C_{\phi}$ given by
\begin{equation}\label{elliptic curve}
\phi(x,y,z)=y^2z-4x^3+g_2xz^2+g_3z^3,
\end{equation}
where $g_2$ and $g_3$ are elements in a field $K$. We want to find the $3\times 3$ square matrices $L,M,N$ such that
$$\det(xL+yM+zN)=\phi(x,y,z).$$
The following determinantal representations of $\phi$:
\begin{equation}\label{matrix}
\left(
  \begin{array}{cccc}
    2x+tz & y+dz & (3t^2-g_2)z \\
    0 & x-tz & y-dz \\
    z & 0 & -2x-tz \\
  \end{array}
\right)
\end{equation}
is obtained from \cite[Section 2]{Vinnikov}, where $t,d\in \overline{K}$ be such that $d^2=4t^3-g_2t-g_3.$ It can be checked that the determinant of \eqref{matrix} is equal to $\phi$.

Now we move to the theory of theta functions to study the case when $K=\mathbb{C}$. The following discussion bases on Wang and Guo \cite{Wang}. In this case, there exists a unique lattice $\Lambda$ coming from the Weierstrass parametrization such that $C_{\phi}(\mathbb{C})\cong \mathbb{C}/\Lambda$. Here $\Lambda=\omega_1 \mathbb{Z} + \omega_2 \mathbb{Z}$ for some $\omega_1,\omega_2\in \mathbb{C}$ with $\tau=\omega_2/\omega_1\in \mathbb{H}$. The two coefficients $g_2$ and $g_3$ of the curve given by $\phi$ can be determined by (see \cite[p. 509]{Wang})
$$g_2=\frac{2}{3}\left(\frac{\pi}{\omega_1}\right)^4(a^8+b^8+c^8),$$
$$g_3=\frac{4}{27}\left(\frac{\pi}{\omega_1}\right)^6 (a^4+b^4)(b^4+c^4)(c^4-a^4),$$
where $a=\theta_{2}(0,\tau)=e^{\frac{\pi i \tau}{4}}\theta(\frac{1}{2}\tau,\tau)$, $b=\theta_{3}(0,\tau)=\theta(0,\tau)$ and $c=\theta_{4}(0,\tau)=\theta(\frac{1}{2},\tau)$ with the even Jacobi theta functions:
$$\theta(z,\tau)=\theta_3(z,\tau):=\sum_{n=-\infty}^{\infty}\exp(\pi i n^2 \tau+2\pi i nz),$$
$$\theta_{2}(z,\tau)=\exp(\pi i \tau/4+\pi i z)\theta(z+\tau/2,\tau),$$
$$\theta_{4}(z,\tau)=\theta(z+1/2,\tau).$$
The above $a,b,c$ are called even theta constants.

Since $(t,d)$ is a point on the affine curve associated to $C_{\phi}$ defined by $\{z\neq 0\}$, it is determined by theta constants via Weierstrass $\mathcal{P}$-function and so are all the coefficients in the linear matrix \eqref{matrix}. To be precise, we consider the Weierstrass $\mathcal{P}$-function associated to the lattice $\Lambda$ defined for all $s\notin \Lambda$ as
$$\mathcal{P}(s)=\mathcal{P}(s;\omega_1,\omega_2):=\frac{1}{s^2}+\sum_{(m,n)\in \mathbb{Z}^2 \setminus (0,0)}\left(\frac{1}{(s+m\omega_1+n\omega_2)^2}-\frac{1}{(m\omega_1+n\omega_2)^2}\right).$$
As in \cite[p. 469]{Wang}, it satisfies the differential equation
$$\mathcal{P}'(s)^2=4\mathcal{P}(s)^3 - g_2 \mathcal{P}(s) - g_3.$$
The point $(t,d)$ on the curve can be parametrized as $t=\mathcal{P}(s)$ and $d=\mathcal{P}'(s)$ for some $s\notin \Lambda$. It is known that the discriminant of the cubic \eqref{elliptic curve} is given by the formula
\begin{equation}\label{disc}
\Delta_{\phi}=2^{12}(g_2^3-27g_3^2)=2^{16}\left(\frac{\pi}{\omega_1}\right)^{12} (abc)^8.
\end{equation}
We will give another proof for the formula \eqref{disc} using resultant and the determinantal representation \eqref{matrix}. From \cite[p. 434]{Gelfand}, the discriminant of a homogeneous cubic polynomial $\phi(x,y,z)$ can be computed by resultant defined there as
\begin{equation}\label{disc res}
\Delta_{\phi}=-\text{Res}(\phi_x,\phi_y,\phi_z)/27.
\end{equation}
The reader can have a look at \cite[Chapter 13]{Gelfand} for a general discussion about resultants. An explicit example of computing the discriminant of cubic curves using resultant is described in \cite[Proposition 2.5, Example 2.7]{Sutherland}.

We choose the minus sign in \eqref{disc res} so that the sign of the discriminant is compatible to other sections of the paper. To simplify the computation, we choose a special value for the Weierstrass function $\mathcal{P}(s)$, namely, the 2-torsion point $s=\omega_2/2$. In this case $\mathcal{P}(\omega_2/2)=-\frac{\pi^2}{3\omega_1^2} (a^4+b^4)$ and $\mathcal{P}'(\omega_2/2)=0$ by \cite[p. 470,509]{Wang}. Then $d=0$ and $t=-\frac{\pi^2}{3\omega_1^2}(a^4+b^4)$. Besides, using the Jacobi's identity $a^4+c^4=b^4$ (see \cite[p. 504]{Wang}), the matrix \eqref{matrix} can be written in the form
\begin{equation}\label{new matrix}
\left(
  \begin{array}{ccc}
    2x-\frac{\pi^2}{3\omega_1^2}(a^4+b^4)z & y & -(\frac{\pi}{\omega_1})^4 c^8 z \\
    0 & x+\frac{\pi^2}{3\omega_1^2}(a^4+b^4)z & y \\
    z & 0 & -2x+\frac{\pi^2}{3\omega_1^2}(a^4+b^4)z \\
  \end{array}
\right).
\end{equation}
Theorem \ref{theorem det rep of elliptic curves} has thus been proved. From the representation $\phi=\det(U)$, where $U$ is given by \eqref{new matrix}, we get that
$$\phi_x=-12x^2+\frac{2}{3}\left(\frac{\pi}{\omega_1}\right)^4 (a^8+b^8+c^8)z^2,$$
$$\phi_y=2yz, \text{  and  }$$
$$\phi_z=y^2+\frac{4}{3}\left(\frac{\pi}{\omega_1}\right)^4(a^8+b^8+c^8)xz+\frac{4}{9}\left(\frac{\pi}{\omega_1}\right)^6(a^4+b^4)(b^4+c^4)(c^4-a^4)z^2.$$
The discriminant $\Delta_{\phi}$ of the cubic $\phi$ is then obtained via \eqref{disc res}
$$\Delta_{\phi}=2^{16}\left(\frac{\pi}{\omega_1}\right)^{12} (abc)^8.$$
In fact, we can directly use \eqref{disc res} to the curve \eqref{elliptic curve}. However, this approach of determinantal representations might be applied to more general cases. More details will be explained in Section \ref{section disc of plane cubic curves}.

\section{Determinantal representations of complex plane curves}\label{section det rep of complex curves}
In Section \ref{section det rep of elliptic curves}, we have already seen that one can compute the discriminant of smooth curves over $\mathbb{C}$ in Weierstrass form by using determinantal representations.  Our goal is to generalize this result to any smooth cubic curve. To do this, one can compute determinantal representations of plane curves of arbitrary degrees based on Theorem 5.1 in \cite{Ball}. We will in this section prove Theorem \ref{theorem det rep of complex curves}. Let us first introduce some notations.

Let $X$ be a compact Riemann surface of positive genus, let $\mathcal{L}$ be a line bundle of half differentials on $X$ (a theta characteristic), i.e., $\mathcal{L}^{\otimes 2}$ is the canonical bundle $\omega_X$ on $X$ and let $\chi$ be a flat line bundle over $X$ such that $h^0(\chi \otimes \mathcal{L})=0$. We associate to $\chi$ the Cauchy kernel $K(\chi;\cdot,\cdot)$ as defined in Section 2 of \cite{Ball}. Let $\lambda_1,\lambda_2$ be two scalar meromorphic functions on $X$, which generate the whole field of meromorphic functions. Assume that all poles of $\lambda_1,\lambda_2$ are simple and labeled as $P_1,...,P_d \in X$. One can write the Laurent expansion of $\lambda_k$ at $P_i$ $(1\leq i \leq d, k=1,2)$ with some fixed local coordinate $t_i=t_i(P)$ centered at $P=P_i$
$$\lambda_k(P)=-\frac{c_{ik}}{t_i}-d_{ik}+O(|t_i|).$$
Here our notation $f=O(g)$ means $|f| \leq A |g|$ for some positive constant $A$. Then we define the $d\times d$ matrices $L,M,N$ by

$$L=\text{diag}_{1\leq i \leq d} (c_{i2}), \text{  } M=\text{diag}_{1\leq i \leq d} (-c_{i1}), \text{ } N=(n_{ij})_{i,j},$$

where
$$n_{ij}=\left\{
                     \begin{array}{ll}
                         d_{i1}c_{i2}-d_{i2}c_{i1}, & \hbox{$i=j$;} \\
        \displaystyle   (c_{i1}c_{j2}-c_{j1}c_{i2})\frac{K(\chi;P_i,P_j)}{dt_j(P_j)}, & \hbox{$i\neq j$.}
                         \end{array}
                     \right.
$$
The result mentioned in \cite{Ball} is the following:

\begin{prop}\label{proposition Ball Vinnikov}
The map $\pi_0 : X \rightarrow \mathbb{C}^2$ given by $\pi_0(P)=(\lambda_1(P),\lambda_2(P))$ maps $X \backslash \{P_1,...,P_d\}$ onto the affine part $C^0$ of an algebraic curve $C\subset \mathbb{P}^2$ and extends to a proper birational map $\pi: X \rightarrow C$ of $X$ in $\mathbb{P}^2$. The defining irreducible homogeneous polynomial $\phi(x,y,z)$ of $C$ is such that (up to multiplying by some constant) $$\phi(x,y,z)=\det(x L+y M+z N).$$
Here the affine part $C^0$ of $C$ is defined by the intersection of $C$ and the complement of the line $z=0$ at infinity.
\end{prop}
The authors in \cite{Ball} prove a more general version of the above proposition where they consider $\chi$ to be any flat vector bundle. We restrict here to the case of line bundle since it is enough for our purpose.

Suppose in this case that $\chi$ is defined by a unitary representation of the fundamental group of $X$ given by
$$\chi(\alpha_i)=\exp(-2\pi i a_i)\text{ and }\chi(\beta_i)=\exp(2\pi i b_i), \text{ } i=1,...,g,$$
where $a_i,b_i\in \mathbb{R}$, $g$ is the genus of $X$ and $\alpha_1,...,\alpha_g,\beta_1,...,\beta_g$ form a symplectic basis of $H_1(X,\mathbb{Z})$. Let $(\eta_1,...,\eta_g)$ be a basis of holomorphic 1-forms on $X$, we form from these bases the period matrix $(\Omega_1 \mid \Omega_2)$ which is the $g\times 2g$-matrix whose entries are
$$(\Omega_1)_{ij}= \int_{\alpha_j}\eta_i \text{   and   } (\Omega_2)_{ij}=\int_{\beta_j} \eta_i, \text{  for  } i,j=1,...,g.$$
We choose the canonical basis $(\eta_1,...,\eta_g)$ of holomorphic 1-forms such that $\int_{\alpha_i}\eta_j =\delta_{ij}$, then the corresponding period matrix will be of the form $(I_g \mid \Omega)$. The matrix $\Omega$ lies in the Siegel upper half space $\mathbb{H}^g$ and it is called the Riemann period matrix of $X$ with respect to the homology basis $\alpha_1,...,\alpha_g,\beta_1,...,\beta_g$. We fix such a symplectic homology basis and the resulting period matrix $\Omega$. Let $J(X)=\mathbb{C}^g/(\mathbb{Z}^g+\Omega \mathbb{Z}^g)$ be the Jacobian of $X$ and $\varphi: X \rightarrow J(X)$ be the Abel-Jacobi map with any fixed base point. An explicit formula for the Cauchy kernel is given in \cite[Theorem 4.1]{Ball} as follows
\begin{equation}\label{Cauchy kernel}
K(\chi;P,Q)=\frac{\theta[\delta](\varphi(Q)-\varphi(P))}{\theta[\delta](0) E(Q,P)},
\end{equation}
where $\theta[\delta]$ is the associated theta function with characteristic $\delta=b+\Omega a=\varphi(\chi)$ $(a=(a_j)_j$ and $b=(b_j)_j)$ and $E(\cdot,\cdot)$ is the prime form on $X\times X$. Recall from \cite[Chapter II]{Fay} that the prime form $E$ is a bi-half-differential with simple zeros along the diagonal of $X\times X$.

Here the theta characteristic $\mathcal{L}$ is chosen such that $\varphi(\mathcal{L})=-\mathcal{K},$ where $\mathcal{K}$ is the vector of Riemann constants. As a consequence of the Riemann singularity theorem, $\theta(b+\Omega a) \neq 0$ if and only if $h^0(\chi \otimes \mathcal{L})=0$. Hence $\theta[\delta](0)\neq 0$ and the formula \eqref{Cauchy kernel} makes sense.

From Proposition \ref{proposition Ball Vinnikov}, one can explicitly provide determinantal representations for complex plane curves using theta functions and the Abel-Jacobi map. The reader can have a look at \cite[Section 4]{Hel Vin}, \cite[Theorem 6]{Plau} or \cite[Theorem 2.2]{Chien Naka} for reference. Note that results in the reference above only apply to the family of hyperbolic curves with a normalization. However, it can be written in the following general form:

\begin{prop}\label{theorem}
Let $C_{\phi}\subset \mathbb{P}^2$ be a non-rational irreducible complex plane curve defined by $\phi=0$, where $\phi(x,y,z)$ is an irreducible homogeneous polynomial of degree $d$. Suppose the $d$ intersection points of $C_{\phi}$ with the line $\{y=0\}$ are distinct non-singular points $P_1,...,P_d$ with coordinates $P_i=(1,0,\beta_i),$ $\beta_i\neq 0$. Then $$\phi(x,y,z)=\lambda\det(xM+yN+zI),$$
where $\lambda=\phi(0,0,1)$, $M=\emph{diag}(-\beta_1,...,-\beta_d)$ and $N=(n_{ij})_{i,j}$ with
$$n_{ii}=-\beta_i \frac{\phi_y(1,0,\beta_i)}{\phi_x(1,0,\beta_i)}$$
and for $i\neq j$
$$
n_{ij}=\frac{\beta_i - \beta_j}{\theta[\delta](0)}\cdot\frac{\theta[\delta](\varphi(P_j)-\varphi(P_i))}{E(P_j,P_i)}\cdot\frac{1}{\sqrt{d(-y/x)(P_i)}\sqrt{d(-y/x)(P_j)}}.
$$
Here $\delta$ is an even theta characteristic such that $\theta[\delta](0)\neq 0$, $\varphi: X \rightarrow J(X)$ is the Abel-Jacobi map from the desingularizing Riemann surface $X$ of $C_{\phi}$ to its Jacobian and $E(.,.)$ is the prime form on $X \times X$.
\end{prop}

Note that the matrix $N$ in the above proposition depends on selections of branches of $\sqrt{d(-y/x)(P_i)}$. However, the determinant $\det(xM + yN + zI)$ is independent of the selections.

We want to generalize Proposition \ref{theorem} in such a way that the line $\{y=0\}$ is replaced by a general line passing through distinct points of $C_{\phi}$.

Let $l$ be a line defined by $\alpha x +\beta y +\gamma z=0$ so that its affine part $l^0$ defined by $ \alpha x + \beta y +\gamma =0$ intersects the affine part $C_{\phi}^0$ of $C_{\phi}$ at $d$ distinct non-singular points $P_i^0$, $i=1,...,d$. Since $\alpha$ and $\beta$ can not be both zero, we can suppose w.l.o.g that $\beta\neq 0$ (the case $\alpha\neq 0$ can be treated similarly). In this case, we can suppose further that $\beta=-1$. Therefore, the line $l$ can be rewritten as $y=\alpha x+\gamma z$. Assume that the intersections points $P_i^0$ of $l^0$ and $C_{\phi}^0$ have non-zero $x$-coordinates so that we can write $P_i^0=(1/\beta_i,\alpha/ \beta_i+\gamma)$ with $\beta_i\neq \beta_j$ if $i\neq j$. Thus the intersection points of $l$ and $C_{\phi}$ are $P_i=(1,\alpha+\gamma\beta_i,\beta_i)$. We now prove the following:

\begin{thm}\label{theorem det rep of complex curves}
Let $C_{\phi}\subset \mathbb{P}^2$ be a non-rational irreducible complex plane curve defined by $\phi=0$, where $\phi(x,y,z)$ is an irreducible homogeneous polynomial of degree $d$. Suppose the $d$ intersection points of $C_{\phi}$ with the line $\{y=\alpha x+\gamma z\}$ are distinct non-singular points $P_1,...,P_d$ with coordinates $P_i=(1,\alpha +\gamma\beta_i,\beta_i),$ $\beta_i\neq 0$. Then up to multiplying by some constant
$$\phi(x,y,z)=\det((M-\alpha N)x+Ny+(I-\gamma N)z),$$
where $M=\emph{diag}(-\beta_1,...,-\beta_d)$ and $N=(n_{ij})_{i,j}$ with
$$n_{ii}=-\frac{\beta_i \phi_y(P_i)}{(\phi_x+\alpha \phi_y)(P_i)}$$
and for $i\neq j$
$$
n_{ij}=\frac{\theta[\delta](\varphi(P_j)-\varphi(P_i))}{\theta[\delta](0)E(P_j,P_i)}\cdot\frac{\beta_i-\beta_j}{\sqrt{\beta_i(\alpha dx - dy)(P_i)}\sqrt{\beta_j(\alpha dx - dy)(P_j)}}.
$$
Here $\delta$ is an even theta characteristic such that $\theta[\delta](0)\neq 0$, $\varphi: X \rightarrow J(X)$ is the Abel-Jacobi map from the desingularizing Riemann surface $X$ of $C_{\phi}$ to its Jacobian and $E(.,.)$ is the prime form on $X \times X$.
\end{thm}
\begin{proof}
Apply Proposition \ref{proposition Ball Vinnikov} with the pair of meromorphic functions on the desingularizing Riemann surface $X$ of $C_{\phi}$:
$$\lambda_1=\frac{1}{y - \alpha x - \gamma}, \text{ }\lambda_2=\frac{x}{y - \alpha x - \gamma}$$
and the local coordinates $t=\frac{\alpha x-y+\gamma}{x}$ at the poles $P_i$ (zeros of $\alpha x -y+\gamma$). The next step is to write down Laurent expansions of $\lambda_1,\lambda_2$ at $P_i$. We have
$$\lambda_2=-1/t \Rightarrow c_{i2}=1, \text{ } d_{i2}=0 \text{ } \forall i.$$
At $P_i$ we have $\lambda_1=-\frac{1}{t}(\frac{1}{x})$. Since
$$\frac{1}{x}=\beta_i+\frac{d(\frac{1}{x})}{d(\frac{\alpha x-y+\gamma}{x})}(P_i)t+O(|t|^2)=\beta_i+\beta_i\frac{dx}{d(y - \alpha x)}(P_i) t+O(|t|^2),$$
one deduces that
$$c_{i1}=\beta_i, \text{ }d_{i1}=\beta_i\frac{dx}{d(y - \alpha x)}(P_i).$$
We then obtain from Proposition \ref{proposition Ball Vinnikov} that (up to some constant)
$$\phi=\det((M-\alpha N)x+Ny+(I-\gamma N)z),$$
where $M=\text{diag}(-\beta_1,...,-\beta_d)$ and $N=(n_{ij})$ with
$$n_{ii}=\beta_i\frac{dx}{d(y - \alpha x)}(P_i)$$
and for $i\neq j$
$$n_{ij}=\frac{\theta[\delta](\varphi(P_j)-\varphi(P_i))}{\theta[\delta](0)E(P_j,P_i)}\cdot\frac{\beta_i-\beta_j}{d(\frac{\alpha x - y + \gamma}{x})(P_j)}.$$
Here $\delta$ is an even theta characteristic with $\theta[\delta](0)\neq 0$. Note that the affine part $C_{\phi}^0$ of $C_{\phi}$ is defined by
$$\{(\lambda_1(P),\lambda_2(P)) \mid P\in X \setminus \{P_1,...,P_d\}\}.$$
Furthermore, if we replace $N$ by the matrix $N'$ which has the same diagonal elements as $N$ but different off-diagonal elements
$$
n_{ij}'=\frac{\theta[\delta](\varphi(P_j)-\varphi(P_i))}{\theta[\delta](0)E(P_j,P_i)}\cdot\frac{\beta_i-\beta_j}{\sqrt{d(\frac{\alpha x - y + \gamma }{x})(P_i)}\sqrt{d(\frac{\alpha x - y + \gamma}{x})(P_j)}},
$$
then the determinantal representation does not change. Indeed, let
$$U=(M-\alpha N)x+Ny+(I-\gamma N)z$$
and
$$U'=(M-\alpha N')x+N'y+(I-\gamma N')z,$$
if we multiply the $i^{th}$-column of $U$ and the $i^{th}$-row of $U'$ (for $i=1,...,d$) with the term $\sqrt{d(\frac{\alpha x - y + \gamma}{x})(P_i)}$ then both of them will become the same matrix $U^*$. Consequently,
$$\det(U)=\det(U')=\frac{\det(U^*)}{\prod_{i=1}^d \sqrt{d(\frac{\alpha x - y + \gamma}{x})(P_i)}}.$$
Observe that $d(\frac{\alpha x - y + \gamma}{x})(P_i)=\beta_i(\alpha dx - dy)(P_i)$ and
$$\frac{dx}{d(y - \alpha x)}(P_i)=-\frac{\phi_y(P_i)}{(\phi_x+\alpha \phi_y)(P_i)}$$
by implicit function theorem with the fact that $(\phi_x+\alpha \phi_y)(P_i)\neq 0$. Indeed, since the polynomial $f(x):=\phi(x,\alpha x + \gamma, 1)$ has distinct roots $(1/\beta_i)$ we conclude that $f'(1/\beta_i)\neq 0$ and hence $(\phi_x+\alpha \phi_y)(P_i)\neq 0$. This completes the proof of Theorem \ref{theorem det rep of complex curves}.
\end{proof}
Proposition \ref{theorem} is then established by reducing to the case $\alpha=\gamma=0$. Theorem \ref{theorem det rep of complex curves} will be applied in the next section to get a formula for the discriminant of plane cubic curves.

\begin{rem}\label{remark alternative lines l}
One can also reformulate the analogous statement to the Theorem \ref{theorem det rep of complex curves} if the line $y=\alpha x + \gamma z$ is replaced by $x=\alpha y + \gamma z$.
\end{rem}

\section{Discriminants of plane cubic curves}\label{section disc of plane cubic curves}

We now study the main object of interest in which we consider a smooth plane curve $C_{\phi}$ over $\mathbb{C}$ defined by the cubic form $\phi$. The affine part of the curve $C_{\phi}$ is parametrized as
$$\{(x,y,1)=(R_1(\mathcal{P}(s),\mathcal{P}'(s)),R_2(\mathcal{P}(s),\mathcal{P}'(s)),1)\},$$
where $\mathcal{P}(s;\omega_1,\omega_2)$ is the Weierstrass $\mathcal{P}$-function associated to some $\omega_1,\omega_2\in \mathbb{C}$ satisfying Im$(\omega_2/\omega_1)>0$. The field of meromorphic functions on a general genus one curve is generated by $\mathcal{P}, \mathcal{P}'$ associated to some periods $\omega_1,\omega_2$. Thus $R_1$ and $R_2$ are rational functions on $\mathcal{P}, \mathcal{P}'$.

In this section, we use the standard notation $\tau$ of genus one case instead of $\Omega$ for the period matrix. Moreover, we use the general Jacobian $\mathbb{C}/(\omega_1 \mathbb{Z} + \omega_2 \mathbb{Z})$ in place of the normalized one $\mathbb{C}/(\mathbb{Z}+\tau \mathbb{Z})$ for $\tau=\omega_2/\omega_1$ in order to use the properties of the function $\mathcal{P}$. By this change, an extra factor $1/\omega_1$ appears in the below elements $n_{ij}$ $(i\neq j)$ in comparing with Theorem \ref{theorem det rep of complex curves}. This idea was mentioned in \cite[Theorem 2.4]{Chien Naka}. In addition, we use the notation $\theta_{\delta}$ $(\delta=1,2,3,4)$ for theta functions as in Section \ref{section det rep of elliptic curves} instead of $\theta[\delta]$.

The prime form $E(P,Q)$ in genus one case is better understood so that a consequence of Theorem \ref{theorem det rep of complex curves} is obtained as follows:
\begin{cor}\label{corollary det rep of plane cubic}
Let $C_{\phi}\subset \mathbb{P}^2$ be a smooth plane cubic curve defined by $\phi=0$, where $\phi(x,y,z)$ is a non-singular homogeneous cubic polynomial. Suppose the line $y=\alpha x+\gamma z$ intersects $C_{\phi}$ at 3 distinct points $P_1,P_2,P_3$ with coordinates $P_i=(1,\alpha +\gamma\beta_i,\beta_i),$ $\beta_i\neq 0$. Then up to multiplying by some constant
\begin{equation}\label{determinant}
\phi(x,y,z)=\det((M-\alpha N)x+Ny+(I-\gamma N)z),
\end{equation}
where $M=\emph{diag}(-\beta_1,-\beta_2,-\beta_3)$ and $N=(n_{ij})_{i,j}$ with
$$n_{ii}=-\frac{\beta_i \phi_y(P_i)}{(\phi_x+\alpha \phi_y)(P_i)}$$
and for $i\neq j$
$$
n_{ij}=\frac{\theta_1'(0)\theta_{\delta}((Q_j-Q_i)/\omega_1)}{\omega_1\theta_{\delta}(0)\theta_1((Q_j-Q_i)/\omega_1)}\cdot\frac{\beta_i-\beta_j}{\sqrt{\beta_i(\alpha R_1' - R_2')(Q_i)}\sqrt{\beta_j(\alpha R_1' - R_2')(Q_j)}}
$$
Here $\delta$ is any even theta characteristic, i.e., $\delta=2,3$ or $4$ and $Q_i=\varphi(P_i)$. Note that we also have an analogous statement of this corollary by Remark \ref{remark alternative lines l}.
\end{cor}

In general, the rational functions $R_1$ and $R_2$ have complicated expressions. However, we have better interpretations in the case of plane cubic curves. In this case, $C_{\phi}$ always has a flex point and hence can be transformed to a Weierstrass form after a linear coordinate change (see \cite[Section 4.4]{Cremona transform}). Thus, we are able to present rational functions $R_1,R_2$ as:
$$R_1(s)=\lambda_{11}\mathcal{P}(s)+\lambda_{12}\mathcal{P}'(s)+\lambda_{13},$$
\begin{equation}\label{R1 R2}
R_2(s)=\lambda_{21}\mathcal{P}(s)+\lambda_{22}\mathcal{P}'(s)+\lambda_{23}.
\end{equation}
The constants $\lambda_{ij}\in \mathbb{C}$ satisfy $\lambda_{11}\lambda_{22}\neq \lambda_{12}\lambda_{21}$ and depend on the coefficients of $\phi$. Here we fix any flex point and the corresponding periods $\omega_1,\omega_2$ coming from the Weierstrass parametrization of the Weierstrass form.

To shorten the determinantal representation, one should look at 2-torsion points to simplify $\theta$ and $\mathcal{P}$. More precisely, we consider the line $l$ which intersects $C_{\phi}$ at the points $P_i$ such that the corresponding points $Q_i$ on the torus $\mathbb{C}/(\omega_1\mathbb{Z}+\omega_2\mathbb{Z})$ are $\omega_1/2, (\omega_1+\omega_2)/2$ and $\omega_2/2$ respectively. Suppose that the $x$-coordinates of $P_i$ are all non-zero. We will treat the case $l$ to have the form $y=\alpha x + \gamma z$ and then make use of Corollary \ref{corollary det rep of plane cubic}. The other case can be treated similarly using Remark \ref{remark alternative lines l}. The choice of 2-torsion points gives us the convenience to work with some computations below. Let $a=\theta_2(0,\tau), b=\theta_3(0,\tau), c=\theta_4(0,\tau)$, where $\tau=\omega_2/\omega_1$, we will prove the following:

\begin{prop}\label{proposition disc of plane cubic}
Let $C_{\phi}\subset \mathbb{P}^2$ be a smooth plane curve defined by $\phi=0$, where $\phi(x,y,z)$ is a non-singular homogeneous cubic polynomial. Suppose the line $y=\alpha x+\gamma z$ intersects $C_{\phi}$ at 3 distinct points $P_1,P_2,P_3$ with coordinates $P_i=(1,\alpha +\gamma\beta_i,\beta_i),$ $\beta_i\neq 0$ so that the corresponding points $Q_i=\varphi(P_i)$ of $P_i$ on the torus $\mathbb{C}/(\omega_1 \mathbb{Z} +\omega_2 \mathbb{Z})$ are $\omega_1/2, (\omega_1+\omega_2)/2$ and $\omega_2/2$ respectively. Denote by $k=\alpha\lambda_{12} - \lambda_{22}$, then we have the following expressions (up to some constant) for the discriminant $\Delta_{\phi}$ of $\phi$
$$\Delta_{\phi}=\frac{\lambda_{11}^6 \omega_1^{24}}{2^8 k^{12} \pi^{24} (abc)^{16}}(\beta_1-\beta_2)^6(\beta_1-\beta_3)^6(\beta_2-\beta_3)^6$$
and
$$\Delta_{\phi}=16\left(\frac{\lambda_{11}^2 \pi \beta_1\beta_2\beta_3}{2 k \omega_1}\right)^{12} (abc)^8.$$
\end{prop}
\begin{proof}
By \cite[p. 470, 509]{Wang}, we have $\mathcal{P}'(Q_i)= 0$ for all $i$ and
$$\mathcal{P}(Q_1)=\frac{\pi^2}{3\omega_1^2}(b^4+c^4), \text{ }\mathcal{P}(Q_2)=\frac{\pi^2}{3\omega_1^2}(a^4-c^4),\text{ } \mathcal{P}(Q_3)=-\frac{\pi^2}{3\omega_1^2}(a^4+b^4).$$
Besides, $\mathcal{P}''(s)=6(\mathcal{P}(s))^2- g_2/2$ with $g_2=\frac{2}{3}(\frac{\pi}{\omega_1})^4 (a^8+b^8+c^8)$ as in \cite[p. 469]{Wang}. Thus
$$\mathcal{P}''(Q_1)=\frac{2\pi^4b^4c^4}{\omega_1^4}, \text{ }\mathcal{P}''(Q_2)=-\frac{2\pi^4a^4c^4}{\omega_1^4}, \text{ }\mathcal{P}''(Q_3)=\frac{2\pi^4a^4b^4}{\omega_1^4}.$$
We also have for each $i$
$$-\frac{\phi_y(P_i)}{(\phi_x+\alpha \phi_y)(P_i)}=\frac{dx}{d(y - \alpha x)}(P_i)=\frac{R_1'}{(R_2' - \alpha R_1')}(Q_i)=\frac{\lambda_{12}}{\lambda_{22} - \alpha \lambda_{12}}.$$
Now, choose $\delta=3$ to simplify the matrix $N$ in Corollary \ref{corollary det rep of plane cubic}. Let $k_1=-\lambda_{12}/k$ and note that $\theta_1'(0)=\pi abc$ as in \cite[p. 507]{Wang}, we have $n_{ii}=k_1\beta_i$ and $n_{13}=n_{31}=0$ as $\theta((1+\tau)/2)=0$. Moreover,
$$n_{12}^2=n_{21}^2=\frac{\pi^2 (abc)^2 \theta^2\left(\frac{\tau}{2}\right) (\beta_1 - \beta_2)^2}{\omega_1^2 k^2 b^2 \theta_1^2 \left(\frac{\tau}{2}\right) \beta_1 \beta_2 \mathcal{P}''(Q_1)\mathcal{P}''(Q_2)}=\frac{\omega_1^6 (\beta_1 - \beta_2)^2}{4k^2 \pi^6 \beta_1\beta_2 b^4 c^8},$$

$$n_{23}^2=n_{32}^2=\frac{\pi^2 (abc)^2 \theta^2\left(\frac{1}{2}\right) (\beta_2 - \beta_3)^2}{\omega_1^2 k^2 b^2 \theta_1^2 \left(\frac{1}{2}\right) \beta_2 \beta_3 \mathcal{P}''(Q_2)\mathcal{P}''(Q_3)}=-\frac{\omega_1^6 (\beta_2 - \beta_3)^2}{4k^2 \pi^6 \beta_2\beta_3 a^8 b^4}.$$
Here we use the fact that (see \cite[p. 502]{Wang})
$$\theta\left(\frac{\tau}{2}\right)=q^{-\frac{1}{8}}a, \theta_1\left(\frac{\tau}{2}\right)=i q^{-\frac{1}{8}}c, \theta\left(\frac{1}{2}\right)=c, \theta_1\left(\frac{1}{2}\right)=a$$
with $q=e^{2 \pi i \tau}$. One has $1/\beta_i=\lambda_{11}\mathcal{P}(Q_i)+\lambda_{13}$ from \eqref{R1 R2} and the fact $R_1(Q_i)=1/\beta_i$. Therefore,
$$\frac{\beta_1-\beta_2}{\beta_1\beta_2}=\lambda_{11}(\mathcal{P}(Q_2)-\mathcal{P}(Q_1))=-\lambda_{11}\frac{\pi^2 c^4}{\omega_1^2},$$
\begin{equation}\label{beta}
\frac{\beta_1-\beta_3}{\beta_1\beta_3}=\lambda_{11}(\mathcal{P}(Q_3)-\mathcal{P}(Q_1))=-\lambda_{11}\frac{\pi^2 b^4}{\omega_1^2},
\end{equation}
$$\frac{\beta_2-\beta_3}{\beta_2\beta_3}=\lambda_{11}(\mathcal{P}(Q_3)-\mathcal{P}(Q_2))=-\lambda_{11}\frac{\pi^2 a^4}{\omega_1^2}.$$
It can be seen from \eqref{beta} that $\lambda_{11}\neq 0$. Similarly we have $\lambda_{21}=\alpha \lambda_{11}$ from the identities $R_2(Q_i)=\alpha/\beta_i+\gamma$. Breaking out the determinant \eqref{determinant}, one gets the following expression for $\phi$ (up to some constant $\lambda$)
$$-\beta_1\beta_2\beta_3 x^3 + 3\beta_1\beta_2\beta_3 k_1 x^2y + (\beta_3 n_{12}^2 + \beta_1 n_{23}^2-3\beta_1\beta_2\beta_3 k_1^2)xy^2 + $$
$$k_1(\beta_1\beta_2\beta_3 k_1^2 - \beta_3 n_{12}^2 - \beta_1  n_{23}^2)y^3 + (\beta_1\beta_2 + \beta_1\beta_3 + \beta_2\beta_3)x^2z - 2k_1 (\beta_1\beta_2 + \beta_1 \beta_3 + \beta_2\beta_3)xyz  $$
\begin{equation}\label{F1}
+(k_1^2(\beta_1\beta_2 + \beta_1\beta_3 + \beta_2\beta_3) - n_{12}^2 - n_{23}^2)y^2z - (\beta_1 + \beta_2 + \beta_3)xz^2 + k_1(\beta_1 + \beta_2 + \beta_3)yz^2 + z^3.
\end{equation}
Consequently, $\text{Res}(\phi_x/\lambda,\phi_y/\lambda,\phi_z/\lambda)=$
$$(-432)  (\beta_2 - \beta_3)^2  (\beta_1 - \beta_2)^2  n_{23}^4  n_{12}^4  (\beta_1 - \beta_3)^6  (\beta_2 n_{12}^2 - \beta_3 n_{12}^2 - \beta_1 n_{23}^2 + \beta_2 n_{23}^2)^2.$$
The term $\beta_2 n_{12}^2 - \beta_3 n_{12}^2 - \beta_1 n_{23}^2 + \beta_2 n_{23}^2$ is equal to
$$\frac{(\beta_1-\beta_2)(\beta_2-\beta_3)\omega_1^6}{4 k^2 \pi^6 b^4} \left(\frac{\beta_1-\beta_2}{c^8\beta_1\beta_2}+\frac{\beta_2-\beta_3}{a^8\beta_2\beta_3}\right)$$
$$=-\frac{\lambda_{11}\omega_1^4(\beta_1-\beta_2)(\beta_2-\beta_3)}{4k^2 \pi^4 a^4 c^4},$$
where the later equality comes from \eqref{beta}. Furthermore,
$$(n_{12}n_{23})^4=\frac{\omega_1^{24}(\beta_1-\beta_2)^4(\beta_2-\beta_3)^4}{2^8 k^8 \pi^{24} (abc)^{16} (\beta_1\beta_2)^2 (\beta_2\beta_3)^2} = \frac{\lambda_{11}^4 \omega_1^{16} (\beta_1-\beta_2)^2 (\beta_2-\beta_3)^2}{2^8 k^8 \pi^{16} a^8 b^{16} c^8}.$$
The later equality again comes from \eqref{beta}. Hence
$$\Delta_{\phi}=-\frac{1}{27}\text{Res}(\phi_x,\phi_y,\phi_z)=\frac{\lambda^{12}\lambda_{11}^6 \omega_1^{24}}{2^8 k^{12} \pi^{24} (abc)^{16}}(\beta_1-\beta_2)^6(\beta_1-\beta_3)^6(\beta_2-\beta_3)^6.$$
An alternative form of the discriminant can be established by using \eqref{beta}:
\begin{equation}\label{disc demo}
\Delta_{\phi}=16\left(\frac{\lambda \lambda_{11}^2 \pi \beta_1\beta_2\beta_3}{2 k \omega_1}\right)^{12} (abc)^8.
\end{equation}
This completes the proof of Proposition \ref{proposition disc of plane cubic}.
\end{proof}

We now simplify the formula \eqref{disc demo} by looking at the relationships between $\lambda, \lambda_{11}, k$ and $\beta_1\beta_2\beta_3$. It can be seen from \eqref{F1} that $\lambda\beta_1\beta_2\beta_3=-\phi(1,0,0)$. The transformation \eqref{R1 R2} means that if we write
$$x=\lambda_{11}X+\lambda_{12}Y+\lambda_{13},$$
$$y=\lambda_{21}X+\lambda_{22}Y+\lambda_{23}$$
then the affine curve $\phi(x,y,1)=0$ will be transformed to a Weierstrass equation \linebreak $-Y^2+4X^3-g_2X-g_3=0$. In addition, the inverse transformation
$$X=l_{11}x+l_{12}y+l_{13},$$
$$Y=l_{21}x+l_{22}y+l_{23}$$
would transform the Weierstrass equation $-Y^2+4X^3-g_2X-g_3=0$ to:
$$4l_{11}^3 x^3 + 12 l_{11}^2 l_{12} x^2y + 12 l_{11} l_{12}^2 xy^2 + 4 l_{12}^3 y^3 + (12 l_{11}^2 l_{13} - l_{21}^2)x^2 + $$
\begin{equation}\label{F2}
(24 l_{11} l_{12} l_{13} - 2 l_{21} l_{22}) xy + (12 l_{12}^2 l_{13} - l_{22}^2) y^2 + (12 l_{11} l_{13}^2 - 2 l_{21} l_{23} - l_{11} g_2)x +
\end{equation}
$$(12 l_{12} l_{13}^2 - 2 l_{22} l_{23} - l_{12} g_2)y + 4 l_{13}^3 - l_{23}^2 - l_{13} g_2 - g_3.$$
One can check that $l_{11}=\lambda_{22}/D, l_{12}=-\lambda_{12}/D, l_{21}=-\lambda_{21}/D$ and $l_{22}=\lambda_{11}/D$ with $D=\lambda_{11}\lambda_{22}-\lambda_{12}\lambda_{21}$. Compare the coefficients of $x^3$ and $x^2 y$ in \eqref{F1} and \eqref{F2}, we have
$$\left\{
    \begin{array}{ll}
      4 l_{11}^3=-\lambda \beta_1\beta_2\beta_3, \\
      12l_{11}^2 l_{12} = 3\lambda \beta_1\beta_2\beta_3 k_1.
    \end{array}
  \right.$$
Or
$$
  \left\{
    \begin{array}{ll}
      4 \lambda_{22}^3=-\lambda \beta_1\beta_2\beta_3 D^3, \\
      12 \lambda_{22}^2 \lambda_{12} = -3\lambda \beta_1\beta_2\beta_3 k_1 D^3.
    \end{array}
  \right.
$$
The second identity shows that $\lambda_{11}=-4\lambda_{22}^2/(\lambda\beta_1\beta_2\beta_3 D^2)$. Hence $\lambda\lambda_{11}^3\beta_1\beta_2\beta_3=-4$ from the first identity. The following result has thus been proved from \eqref{disc demo}.
\begin{thm}\label{theorem general cubic theta}
Let $C_{\phi}$ be a smooth plane cubic curve as in Proposition \ref{proposition disc of plane cubic}. Then the discriminant $\Delta_{\phi}$ of $\phi$ satisfies
$$\Delta_{\phi}=\frac{2^{16}}{(\lambda_{11}\lambda_{22}-\lambda_{12}\lambda_{21})^{12}}\left(\frac{\pi}{\omega_1}\right)^{12}(abc)^8.$$
\end{thm}
Let us look at the example when $\phi$ is given in the Weierstrass form $-y^2+4x^3-g_2x-g_3$. In this case, $\lambda_{11}=\lambda_{22}=1$ and $\lambda_{12}=\lambda_{21}=0$. We thus recover the classical formula $\Delta_{\phi}=2^{16}(\frac{\pi}{\omega_1})^{12}(abc)^8$.

From Remark \ref{remark alternative lines l}, one can also treat the other case where the line $l$ passes through 2-torsion points of $C_{\phi}$. Furthermore, the set of cubics $\phi$ in the above theorem forms an open dense subset of the space of all ternary cubics and we have thus obtained Theorem \ref{theorem intro disc cubic theta}.

\section{Plane quartics and Klein's formula}\label{section Klein}

In this section, we provide an overview to a beautiful formula of Klein on plane quartics, which are non-hyperelliptic curves of genus three. More concretely, let $C_F$ be a smooth plane curve over $\mathbb{C}$ given by a quartic $F$, let $\alpha_1,\alpha_2,\alpha_3,\beta_1,\beta_2,\beta_3$ be a symplectic basis of $H_1(C_F,\mathbb{Z})$ and let $\eta_1,\eta_2,\eta_3$ be the classical basis of holomorphic 1-forms of $\Omega_{\mathbb{C}}^1(C_F)$ defined in \cite[p. 329]{Ritzenthaler}. We construct from these the period matrix $[\Omega_1 \text{ } \Omega_2]$ whose entries are
$$(\Omega_1)_{ij}= \int_{\alpha_i}\eta_j \text{   and   } (\Omega_2)_{ij}=\int_{\beta_i} \eta_j, \text{  for  } i,j=1,2,3.$$
Denote by $\tau=\Omega_2^{-1}\Omega_1$, the discriminant $\Delta_F$ of $F$ satisfies the following formula (see \cite[p. 72]{Klein}, \cite[Theorem 2.2.3]{Ritzenthaler}):
\begin{equation}\label{Klein}
\Delta_F^2=\frac{2^{26}\pi^{54}}{(\det \Omega_2)^{18}}\prod_{\delta \text{ even }} \theta_{\delta}(0,\tau).
\end{equation}
Here $\theta_{\delta}$ is the Riemann theta function with characteristic $\delta=(\delta_1,\delta_2)$, where \linebreak $\delta_1,\delta_2\in \{0,1\}^3$, defined for any $z\in \mathbb{C}^3$ as:
$$\theta_{\delta}(z,\tau)=\sum_{n\in \mathbb{Z}^3}\exp 2\pi i \left(\frac{1}{2}(n+\delta_1)^t\tau(n+\delta_1)+(n+\delta_1)^t(z+\delta_2)\right).$$
The product in \eqref{Klein} runs over all 36 even theta characteristics of genus three. The characteristic $\delta$ is called even if the corresponding theta function $\theta_{\delta}$ is an even function in $z$. This formula should be compared with \eqref{Disc Weierstrass} in the cubic case. One can ask if it is possible to use determinantal representations to establish the formula \eqref{Klein} or not? For this, the authors in \cite[Corollary 5.3]{Manni} have obtained determinantal representations for plane quartics described by theta constants. More concretely, they use plane quartics' bitangents to construct the representations. Thus it might be interesting to explore the problem in this case.

\section*{Acknowledgement}

The author would like to thank Professor Dennis Eriksson for introducing him to the topic as well as providing him with important ideas, corrections and comments.

\end{document}